\documentclass{amsart}
\usepackage{amsmath,amssymb}
\usepackage{bm}
\usepackage{graphicx}
\usepackage{float}
\usepackage{ascmac}
\usepackage{fancybox}
\usepackage{tabularx}
\usepackage{multicol}
\usepackage{enumerate}
\usepackage{slashed}
\usepackage{mathrsfs}

\newtheorem{dfn}{Definition}[section]
\newtheorem{thm}[dfn]{Theorem}
\newtheorem{prop}[dfn]{Proposition}

\newtheorem{cor}[dfn]{Corollary}
\newtheorem{remark}[dfn]{Remark}

\numberwithin{equation}{section}

\DeclareMathOperator*{\divergence}{div}

\title[An inverse obstacle problem for the magnetic Schr\"odinger equation]{An inverse obstacle problem for the magnetic Schr\"odinger equation}

\author{Mourad Choulli}
\address{Universit\'{e} de Lorraine, 34 cours L\'{e}opold, 54052 Nancy cedex, France}
\email{mourad.choulli@univ-lorraine.fr}

\author{Hiroshi Takase}
\address{Institute of Mathematics for Industry, Kyushu University, 744 Motooka, Nishi-ku, Fukuoka 819-0395, Japan}
\email{htakase@imi.kyushu-u.ac.jp}

\date{\today}

\begin{document}
\begin{abstract}
We establish stability inequalities of an inverse obstacle problem for the magnetic Schr\"odinger equation. We mainly study the problem of reconstructing an unknown function defined on the obstacle boundary from two measurements performed on the boundary of a domain surrounding the obstacle. We show for the inverse problem a Lipschitzian stability locally in time and a logarithmic stability globally in time.
\end{abstract}
\maketitle

\section{Introduction}

Let $n\ge 2$ be an integer. Throughout this text, we use the Einstein summation convention for quantities with indices. If a term appears twice, both as a superscript and a subscript, that term is assumed to be summed from $1$ to $n$. From now on, all functions are assumed to be complex-valued unless otherwise stated.

Let $(g_{k\ell})\in W^{3,\infty}(\mathbb{R}^n;\mathbb{R}^{n\times n})$ be a symmetric matrix-valued function satisfying
\[
g_{k\ell}(x)\xi^k\xi^\ell\ge \kappa|\xi|^2\quad x,\xi \in\mathbb{R}^n,
\]
where $\kappa >0$ is a constant. Note that $(g^{k\ell})$  the matrix inverse to $g$ is uniformly positive definite as well. We recall that the Laplace-Beltrami operator associated with  the metric tensor $g=g_{k\ell}dx^k\otimes dx^\ell$ is given by
\[
\Delta_g u:=(1/\sqrt{|g|})\partial_k\left(\sqrt{|g|}g^{k\ell}\partial_\ell u\right),
\]
where $|g|:=\mbox{det} (g)$, and the magnetic Laplace-Beltrami operator with a magnetic potential $a=a_k dx^k$ with $(a_1,\ldots,a_n)\in W^{3,\infty}(\mathbb{R}^n;\mathbb{R}^n)$ is given by
\begin{align*}
Lu:&=(1/\sqrt{|g|})(\partial_k+i a_k)\left(\sqrt{|g|}g^{k\ell}(\partial_\ell +ia_\ell)u\right)
\\
&=\Delta_g u+2i g^{k\ell}a_k \partial_\ell u+g^{k\ell}\left[i(\partial_k a_\ell-\Gamma_{k\ell}^m a_m)-a_k a_\ell\right]u.
\end{align*}
Here, $\Gamma_{k\ell}^m$ denote the Christoffel symbols for $1\le k,\ell,m\le n$, which are given as follows
\[
\Gamma_{k\ell}^m:=(1/2)g^{jm}(\partial_kg_{j \ell}+\partial_\ell g_{jk}-\partial_jg_{k\ell}).
\]

For convenience, we recall the following notations 
\begin{align*}
&\langle X,Y\rangle=g_{k\ell}X^kY^\ell,\quad X=X^k\frac{\partial}{\partial x^k},\; Y=Y^k\frac{\partial}{\partial x^k},
\\
&\nabla_gw=g^{k\ell}\partial_k w\frac{\partial}{\partial x^\ell},
\\
&|\nabla_gw|_g^2=\langle\nabla_gw,\nabla_g\overline{w}\rangle=g^{k\ell}\partial_k w\partial_\ell \overline{w}.
\end{align*}
As usual, the  Hessian matrix $\nabla_g^2w$ is defined by
\[
(\nabla_g^2w)_{k\ell}:= \partial_{k}\partial_\ell w-\Gamma_{k\ell}^m\partial_mw.
\]

Let $dV_g=\sqrt{|g|}dx$, where $dx$ is the Lebesgue measure on $\mathbb{R}^n$. For $O$ an open set of $\mathbb{R}^n$, we assume hereinafter that $L^2(O)$, $H^1(O)$ and $H^2(O)$ are respectively equipped with the norms
\begin{align*}
&\|w\|_2:=\left(\int_O |w|^2dV_g\right)^{1/2},
\\
&\|w\|_{1,2}:= \left(\int_O \left[|w|^2+|\nabla_gw|_g^2\right]dV_g\right)^{1/2},
\\
&\|w\|_{2,2}:= \left(\int_O \left[|w|^2+|\nabla_gu|_g^2+|\nabla_g^2w| ^2_g\right]dV_g\right)^{1/2}.
\end{align*}
Here
\[
|\nabla_g^2w| ^2_g:=g^{k_1\ell_1}g^{k_2\ell_2}(\nabla_g^2w)_{k_1k_2}(\nabla_g^2\overline{w})_{\ell_1\ell_2}.
\]
In  this text, $\phi\in C^4(\mathbb{R}^n;\mathbb{R})$ satisfies the following properties
\begin{equation}\label{phi}
\inf_{\{\phi> 0\}}|\nabla_g\phi |_g>0 ,\quad \inf_{\{\phi> 0\}}\inf_{|\xi|=1}\nabla_g ^2\phi(\xi,\xi)>0
\end{equation}
and $\{\phi<0\}\subset\mathbb{R}^n$ is bounded. An example of such a function was given \cite[Example 1.1]{CT2024a}.

Next, let $T>0$, $B:=\{\phi<0\}$, $\Gamma:=\{ \phi=0\}$ ($=\partial B$), $U:=\mathbb{R}^n\setminus \overline{B}$ and $\Sigma :=\Gamma\times (0,T)$. We define
\[Pu:=i\partial_t u+Lu\]
and we consider the non-homogenuous IBVP
\begin{equation}\label{exIBVP}
\begin{cases}
Pu=0\quad \text{in}\; U\times(0,T),
\\
u_{|\Sigma}=f,
\\
u(\cdot,0)=u_0.
\end{cases}
\end{equation}

Let $\mathcal{H}:=W^{3,1}((0,T);L^2(U))\cap W^{2,1}((0,T);H^2(U))$ and
\[
\mathcal{H}_0:=\{H\in \mathcal{H};\; PH(\cdot,0)\in H_0^1(U)\cap H^2(U)\}.
\]
Then consider the trace space
\[
\mathcal{T}:=\{(f,u_0)=(H_{|\Sigma},H(\cdot,0));  H\in \mathcal{H}_0\},
\]

Let $\mathcal{X}:=C^2([0,T];L^2(U))\cap C^1([0,T];H^2(U))$. We prove in Appendix \ref{appA} that, for all $(f,u_0)\in \mathcal{T}$, \eqref{exIBVP} admits a unique solution $u=u(f,u_0)\in \mathcal{X}$.

Let $\Omega \Supset B$ be an arbitrarily fixed bounded $C^{0,1}$ domain such that $D:=\Omega\setminus \overline{B}$ is connected. We seek to establish stability inequalities for the problem of determining the unknown function $f$ from the measurements $u(f,u_0)_{|\Sigma_0}$ and $\partial_{\nu_g} u(f,u_0)_{|\Sigma_0}$, where $\Sigma_0 :=\partial\Omega \times(0,T)$. This problem can be seen as determining an unknown source emanating from the obstacle from two measurements made far from the obstacle.

To the authors' knowledge, this is the first paper to address such an inverse problem for the magnetic Schr\"odinger equation. The same type of inverse problem has been considered by Choulli-Takase \cite{CT2024} for elliptic and parabolic equations, and by Choulli-Takase \cite{CT2024a} for hyperbolic equations.  

We fix $H_0\in \mathcal{H}_0$ so that $\alpha:=\|H_0(\cdot,0)_{|\Gamma}\|_{L^2(\Gamma)}>0$,
where
\[
\|w\|_{L^2(\Gamma)}:=\left(\int_\Gamma|w|^2dS_g\right)^{1/2}.
\]
Here $dS_g:=\sqrt{|g|}dS$, where $dS$ denotes the surface measure on $\partial D$.

Define
\[
\mathscr{F}_0:=\left\{f\in H_{|\Sigma};\; H\in \mathcal{H}_0,\; H(\cdot,0)=H_0(\cdot,0)\right\}.
\]

By this definition, for all $f\in\mathscr{F}_0$, \eqref{exIBVP} has a unique solution $u=u(f)\in \mathcal{X}$.

The following additional notations will be used in the sequel
\begin{align*}
&\nu_g=(\nu_g)^k\frac{\partial}{\partial x^k},\quad (\nu_g)^k=\frac{g^{k\ell}\nu_\ell}{\sqrt{g^{\ell_1 \ell_2}\nu_{\ell_1}\nu_{\ell_2}}},
\\
&\partial_{\nu_g}w=\langle\nu_g,\nabla_g w\rangle,
\end{align*}
where $\nu$ denotes the outer unit normal to $\partial D$. Also, define  $\nabla_{\tau_g} w$ by
\[
\nabla_{\tau_g} w=\nabla_g w-(\partial_{\nu_g} w)\nu_g.
\]
We verify that 
\[
|\nabla_g w|_g^2=|\nabla_{\tau_g} w|_g^2+|\partial_{\nu_g} w|^2.
\]

Let $\beta>0$ be fixed and define the admissible set of the unknown function $f$ by
\begin{align*}
&\mathscr{F}:=\left\{f\in \mathscr{F}_0;\; \|f(\cdot,t)\|_{L^2(\Gamma)}\ge\alpha\quad\text{and}\right.
\\
&\hskip 4cm \left.\|\partial_t f(\cdot,t)\|_{L^2(\Gamma)}+\|\nabla_{\tau_g}f(\cdot,t)\|_{L^2(\Gamma)}\le \beta,\; t\in[0,T]\right\},
\end{align*}
where 
\[
\|\nabla_{\tau_g}w\|_{L^2(\Gamma)}:=\left(\int_\Gamma|\nabla_{\tau_g}w|_g^2dS_g\right)^{1/2}. 
\]

$L^2(\Sigma)$ and $H^1(\Sigma)$ will be endowed respectively with their naturals norms :
\begin{align*}
&\|w\|_{L^2(\Sigma)}:=\left(\int_\Sigma |w|^2dS_gdt\right)^{1/2},
\\
&\|w\|_{H^1(\Sigma)}:=\left(\int_\Sigma \left[|w|^2+|\partial_tw|^2+|\nabla_{\tau_g}w|_g^2\right]dS_gdt\right)^{1/2},
\end{align*}
and similarly for $L^2(\Sigma_0)$ and $H^1(\Sigma_0)$.

The main result of this paper is  the following theorem.
\begin{thm}\label{exLipschitz}
Let $\zeta=(g,\phi,\Omega, a,T,\beta/\alpha)$. Then there exist $C=C(\zeta)>0$ and $c=c(\zeta)>0$ such that for any $\varepsilon\in(0,T/2)$ and $f\in\mathscr{F}$ we have
\[
\|f\|_{H^1(\Gamma\times(\varepsilon,T-\varepsilon))}
\le Ce^{c/\varepsilon}\left(\|u(f)\|_{H^1(\Sigma_0)}+\|\partial_{\nu_g} u(f)\|_{L^2(\Sigma_0)}\right).
\]
\end{thm}

Note that Theorem \ref{exLipschitz} quantifies the unique determination of $f$ from the Cauchy data $(u(f)_{|\Sigma_0},\partial_{\nu_g}u(f)_{|\Sigma_0})$. To our knowledge, there are no other results in the literature comparable to those established in the present work.

In the case where $f$ is written as $f(x,t)=\mathfrak{a}(x)\mathfrak{b}(t)$, where $\mathfrak{b}$ is known, we have the following consequence of Theorem \ref{exLipschitz}.

\begin{cor}
Let $\zeta=(g,\phi,\Omega,a,T,\beta/\alpha)$. Then for all $f=\mathfrak{a}\otimes \mathfrak{b}\in\mathscr{F}$ we have
\[
\|\mathfrak{a}\|_{H^1(\Gamma)}
\le C\left(\|u(f)\|_{H^1(\Sigma_0)}+\|\partial_{\nu_g} u(f)\|_{L^2(\Sigma_0)}\right),
\]
where $C=C(\zeta,\mathfrak{b})>0$ is a constant. 
\end{cor}
\begin{proof}
Since $|\mathfrak{b}|\|\mathfrak{a}\|_{L^2(\Gamma)}\ge \alpha$, we have $\eta:=\min_{[0,T]}|\mathfrak{b}|>0$. By applying Theorem \ref{exLipschitz} with $\varepsilon =T/4$, we get
\[
\eta\sqrt{T/2}\left(\|\mathfrak{a}\|_{L^2(\Gamma)}+\|\nabla_{\tau_g}\mathfrak{a}\|_{L^2(\Gamma)}\right)
\le Ce^{4c/T}\left(\|u(f)\|_{H^1(\Sigma_0)}+\|\partial_{\nu_g} u(f)\|_{L^2(\Sigma_0)}\right),
\]
where $C$ and $c$ are as in the statement of Theorem \ref{exLipschitz}.
\end{proof}

Theorem \ref{exLipschitz} gives a stability inequality locally in time since the time intervals close to $t=0$ and $t=T$ are truncated. Globally in time, we obtain the following interpolation inequality.

\begin{cor}
Let $\zeta=(g,\phi,\Omega,a,T,\beta/\alpha)$ and $r\in(0,1/2)$. Then there exist $C=C(\zeta,r)>0$ and $c=c(\zeta,r)>0$ such that for any $\varepsilon\in(0,1)$ and $f\in\mathscr{F}$ we have
\begin{equation}\label{interpolation}
\|f\|_{L^2(\Gamma\times(0,T))}\le C\left(e^{c/\varepsilon}\mathcal{D}(f)+\varepsilon^{r}\|f\|_{H^1((0,T);L^2(\Gamma))}\right),
\end{equation}
where
\[
\mathcal{D}(f):=\|u(f)\|_{H^1(\Sigma_0)}+\|\partial_{\nu_g} u(f)\|_{L^2(\Sigma_0)}.
\]
\end{cor}
\begin{proof}
By means of the Hardy-type inequality in \cite[Corollary 3.1]{Choulli2021} (see also \cite[Corollary 1.16]{Choulli}), for any $r\in(0,1/2)$ there exists a constant $C>0$, we have for any $\varepsilon\in(0,T/2)$,
\[
\|f\|_{L^2((0,\varepsilon);L^2(\Gamma))}+\|f\|_{L^2((T-\varepsilon,T);L^2(\Gamma))}\le C\varepsilon^{r}\|f\|_{H^1((0,T);L^2(\Gamma))}.
\]
This inequality combined with that in Theorem \ref{exLipschitz} gives \eqref{interpolation} with $\varepsilon\in(0,T/2)$. We complete the proof by replacing $\varepsilon$ by $2\varepsilon/T$.
\end{proof}

\begin{remark}
{\rm
By minimizing the right-hand side of \eqref{interpolation} with respect to $\varepsilon\in(0,1)$, we obtain a stability inequality with logarithmic modulus of continuity.
}
\end{remark}

\section{Proof of Theorem \ref{exLipschitz}}

The main ingredient of the proof of theorem \ref{exLipschitz} is a Carleman inequality. Before stating this inequality precisely, we recall that $D=\Omega \setminus \overline{B}$, $\Gamma =\partial B$, $\Sigma=\Gamma\times (0,T)$, $\Sigma_0=\partial\Omega\times(0,T)$ and
\begin{align*}
Pu&=i\partial_t u+Lu
\\
&=i\partial_t u+\Delta_g u+2i g^{k\ell}a_k \partial_\ell u+g^{k\ell}\left[i(\partial_k a_\ell-\Gamma_{k\ell}^m a_m)-a_k a_\ell\right]u,
\end{align*}
where $a=a_k dx^k$ with $(a_1,\ldots,a_n)\in W^{3,\infty}(\mathbb{R}^n;\mathbb{R}^n)$. Let $Q:=D\times (0,T)$. Set $\ell(t):=[t(T-t)]^{-1}$ and
\begin{align*}
&\varphi(x,t):=(e^{\gamma(\phi(x)+2m)}-e^{4\gamma m})\ell(t),
\\
& \xi(x,t):=e^{\gamma(\phi(x)+2m)}\ell(t),
\end{align*}
where $\phi\in C^4(\mathbb{R}^n)$ satisfies \eqref{phi}.

\begin{prop}\label{global_Carleman_estimate}
Let $\zeta_0:=(g,\phi,a,T)$ and $\sigma:=s\gamma \xi$. There exist $\gamma_\ast=\gamma_\ast(\zeta_0)>0$,  $s_\ast=s_\ast (\zeta_0)>0$ and $C=C(\zeta_0)>0$ such that for any $\gamma\ge\gamma_\ast$, $s\ge s_\ast$ and $u\in L^2((0,T);H^2(D))\cap H^1((0,T);H^1(D))$ we have
\begin{align*}
&C\left(\int_{Q} e^{2s\varphi}\sigma(|\nabla_g u|_g^2+\gamma\sigma^2|u|^2)dV_gdt+\int_{\Sigma}e^{2s\varphi}\sigma(|\partial_{\nu_g}u|^2+\sigma^2|u|^2)dS_gdt\right)
\\
&\hskip 1cm\le \int_{Q} e^{2s\varphi}|Pu|^2dV_gdt+\int_{\Sigma}e^{2s\varphi}\left(s(\gamma\xi)^{-1}|\partial_t u|^2+\sigma|\nabla_{\tau_g} u|_g^2\right)dS_gdt
\\
&\hskip 4cm +\int_{\Sigma_0} e^{2s\varphi}\left(\sigma^{-1}|\partial_t u|^2+\sigma|\nabla_g u|_g^2+\sigma^3|u|^2\right)dS_gdt.
\end{align*}
\end{prop}

The proof of Proposition \ref{global_Carleman_estimate} is quite technical but for sake of completeness we give its detailed proof in Section \ref{proof_Carleman_estimate}.

It is worth noting that only the integral term for the time derivative and the tangential derivative of $u$ on $\Sigma$ appears in the right-hand side of the above inequality. To our knowledge, this is the first Carleman estimate of this type for the dynamical Schr\"odinger operator. See also Baudouin-Puel \cite{Baudouin2002}, Mercado-Osses-Rosier \cite{Mercado2008}, Huang-Kian-Soccorsi-Yamamoto \cite{Huang2019} and Mercado-Morales \cite{Mercado2023} for Carleman estimates with degenerate weight functions such as $\varphi$.

\begin{proof}[Proof of Theorem \ref{exLipschitz}]
Let $f\in\mathscr{F}$ and $u=u(f)$. Then we have
\begin{align}\label{Sc}
&\|\partial_t u(\cdot,t)\|_{L^2(\Gamma)}+\|\nabla_{\tau_g} u(\cdot,t)\|_{L^2(\Gamma)}
\\
&\hskip 1cm =\|\partial_t f(\cdot,t)\|_{L^2(\Gamma)}+\|\nabla_{\tau_g} f(\cdot,t)\|_{L^2(\Gamma)}\le (\beta/\alpha)\|f(\cdot,t)\|_{L^2(\Gamma)}.\notag
\end{align}

Henceforth, $C=C(\zeta)>0$ and $c=c(\zeta)>0$ denote generic constants. Fix $\gamma\ge\gamma_\ast$, where $\gamma_\ast$ is given by Proposition \ref{global_Carleman_estimate}, and set $\omega:=s\xi$. By \eqref{Sc}, $\varphi_{|\Sigma}=(e^{2\gamma m}-e^{4\gamma m})\ell$ and $\xi_{|\Sigma}=e^{2\gamma m}\ell$. By applying Proposition \ref{global_Carleman_estimate} to $u$, we obtain
\begin{align*}
&C\int_\Sigma e^{2s\varphi}\omega^3|f|^2dS_gdt
\\
&\hskip 1cm \le \int_{\Sigma_0} e^{2s\varphi}\omega^3(|\partial_t u|^2+|\nabla_g u|_g^2+|u|^2)dS_gdt
\\
&\hskip 3cm +(\beta/\alpha)^2\int_\Sigma e^{2s\varphi}\omega|f|^2dS_gdt,\quad s\ge s_\ast,
\end{align*}
where $s_\ast$ is as in Proposition \ref{global_Carleman_estimate}. Upon modifying $s_\ast$, we  may and do assume that $C\omega^3-(\beta/\alpha)^2\omega>(C/2)\omega^3$. In this case, we have
\[
C\int_\Sigma e^{2s\varphi}\omega^3|f|^2dS_gdt\le \int_{\Sigma_0} e^{2s\varphi}\omega^3(|\partial_t u|^2+|\nabla_g u|_g^2+|u|^2)dS_gdt,\quad s\ge s_\ast.
\]
Let $\varepsilon\in (0,T/2)$ be arbitrarily fixed. Since
\[
\ell(t)^{-1}=t(T-t)\ge T\varepsilon/2,\quad  t\in(\varepsilon,T-\varepsilon),
\]
we obtain
\[
\varphi(t,x)\ge 2(e^{\gamma(\phi(x)+2m)}-e^{4\gamma m})/(T\varepsilon)\ge -c/\varepsilon,
\]
which implies that
\[
e^{2s\varphi}\ge e^{-cs/\varepsilon}\quad \mbox{in}\; \Gamma\times(\varepsilon,T-\varepsilon).
\]
Therefore, it follows that
\[
\int_\Sigma e^{2s\varphi}\omega^3|f|^2dS_gdt  \ge Cs^3e^{-cs/\varepsilon}\|f\|_{L^2(\Gamma\times(\varepsilon,T-\varepsilon))}^2.
\]
Moreover, since
\[
e^{2s\varphi}\omega ^3\le e^{9\gamma m}s^3\ell^3e^{-2cs\ell}\le Ce^{-cs\ell},
\]
we find
\[\int_{\Sigma_0}e^{2s\varphi}\omega^3\left(|\partial_t u|^2+|\nabla_g u|_g^2+|u|^2\right)dS_gdt
\le C\left(\|u\|_{H^1(\Sigma_0)}^2+\|\partial_{\nu_g} u\|_{L^2(\Sigma_0)}^2\right).
\]
Combining these estimates yields
\[
\|f\|_{L^2(\Gamma\times(\varepsilon,T-\varepsilon))}^2\le Ce^{cs/\varepsilon}\left(\|u\|_{H^1(\Sigma_0)}+\|\partial_{\nu_g} u\|_{L^2(\Sigma_0)}\right)^2.
\]
We complete the proof using \eqref{Sc} again.
\end{proof}

\section{Proof of Proposition \ref{global_Carleman_estimate}}\label{proof_Carleman_estimate}

\begin{proof}
Due to the large parameters $\gamma$ and $s$, it  is enough to establish the expected inequality when $a=0$. For simplicity, $\partial_tw$ will denoted by $w'$. 

Let $u\in L^2((0,T);H^2(D))\cap H^1((0,T);H^1(D))$, $z:=e^{s\varphi}u$ and 
\[
P_sz:=e^{s\varphi}(i\partial_t+\Delta_g)(e^{-s\varphi}z).
\]
A direct calculation yields $P_sz=P_s^+z+P_s^- z$, where
\[
\begin{cases}
&P_s^+z:=iz'+\Delta_g z+s^2|\nabla_g\varphi|_g^2z,
\\
&P_s^-z:=-2s\langle\nabla_g\varphi,\nabla_g z\rangle-s\Delta_g\varphi z-is\varphi' z.
\end{cases}
\]

In what follows, all  integrals on $Q=D\times(0,T)$ are with respect to the measure $dV_gdt$ and all those on $\Upsilon:=\partial D\times(0,T)$ are with respect to the surface measure $dS_gdt$. Set 
\[
\Re (P_s^+z,\overline{P_s^-z})_g=:\sum_{k=1}^9 I_k,
\]
where
\begin{align*}
&I_1:=\Re\left(-2i\int_{Q}sz'\langle\nabla_g\varphi,\nabla_g \bar{z}\rangle\right),
\\
&I_2:=\Re\left(-i\int_{Q} sz'\Delta_g\varphi\bar{z}\right),
\\
&I_3:=\Re\left(-\int_Q sz'\varphi' \bar{z}\right)
\\
&I_4:=\Re\left(-2\int_{Q} s\Delta_g z\langle\nabla_g\varphi,\nabla_g \bar{z}\rangle\right),
\\
&I_5:=\Re\left(-\int_{Q} s\Delta_g z\Delta_g\varphi \bar{z}\right),
\\
&I_6:=\Re\left(i\int_Q s\Delta_g z\varphi'\bar{z}\right),
\\
&I_7:=\Re\left(-2\int_{Q} s^3|\nabla_g\varphi|_g^2 z\langle\nabla_g\varphi,\nabla_g \bar{z}\rangle\right),
\\
&I_8:=\Re\left(-\int_{Q} s^3|\nabla_g\varphi|_g^2\Delta_g\varphi |z|^2\right)=-\int_{Q} s^3|\nabla_g\varphi|_g^2\Delta_g\varphi |z|^2,
\end{align*}
and
\[
I_9:=\Re\left(i\int_Q s^3|\nabla_g\varphi|_g^2\varphi' |z|^2\right)=0.
\]

Integrations by parts yield
\begin{align*}
I_1&=\Re\left(2i\int_{Q}sz\langle\nabla_g\varphi,\nabla_g \bar{z}'\rangle\right)+\Re\left(2i\int_{Q}sz\langle\nabla_g\varphi',\nabla_g \bar{z}\rangle\right)
\\
&=\Re\left(-2i\int_{Q}s \bar{z}'\langle\nabla_g\varphi,\nabla_g z\rangle\right)+\Re\left(-2i\int_{Q}s \bar{z}' \Delta_g\varphi z\right)
\\
&\hskip 2cm +\Re\left(2i\int_\Upsilon s z\partial_{\nu_g}\varphi\bar{z}'\right)+\Re\left(2i\int_{Q}sz\langle\nabla_g\varphi',\nabla_g \bar{z}\rangle\right)
\\
&=-\Re\left(\overline{-2i\int_{Q}s z'\langle\nabla_g\varphi,\nabla_g \bar{z}\rangle}\right)-2\Re\left(\overline{-i\int_{Q}s z' \Delta_g\varphi \bar{z}}\right)
\\
&\hskip 2cm +\Re\left(2i\int_\Upsilon s z\partial_{\nu_g}\varphi\bar{z}'\right)+\Re\left(2i\int_{Q}sz\langle\nabla_g\varphi',\nabla_g \bar{z}\rangle\right)
\\
&=-I_1-2I_2+\Re\left(2i\int_\Upsilon s z\partial_{\nu_g}\varphi\bar{z}'\right)+\Re\left(2i\int_{Q}sz\langle\nabla_g\varphi',\nabla_g \bar{z}\rangle\right),
\end{align*}
which implies
\[
I_1+I_2=\Re\left(i\int_\Upsilon s z\partial_{\nu_g}\varphi\bar{z}'\right)+\Re\left(i\int_{Q}sz\langle\nabla_g\varphi',\nabla_g \bar{z}\rangle\right).
\]
Moreover, we have
\[
I_3=-(1/2)\int_Q s\varphi'\partial_t(|z|^2)=(1/2)\int_Q s\varphi''|z|^2,
\]
\begin{align*}
I_4&=\Re\left(2\int_{Q} s\nabla_g^2\varphi(\nabla_g z,\nabla_g \bar{z})\right)+\int_Q s\langle\nabla_g\varphi,\nabla_g|\nabla_g z|_g^2\rangle
\\
&\hskip 6cm-\Re\left(2\int_\Upsilon s\partial_{\nu_g}z\langle\nabla_g\varphi,\nabla_g\bar{z}\rangle\right)
\\
&=\Re\left(2\int_{Q} s\nabla_g^2\varphi(\nabla_g z,\nabla_g \bar{z})\right)-\int_Q s\Delta_g\varphi|\nabla_g z|_g^2
\\
&\hskip 3cm -\Re\left(2\int_\Upsilon s\partial_{\nu_g}z\langle\nabla_g\varphi,\nabla_g\bar{z}\rangle\right)+\int_\Upsilon s\partial_{\nu_g}\varphi|\nabla_g z|_g^2,
\end{align*}
\begin{align*}
I_5&=\int_{Q} s\Delta_g\varphi|\nabla_g z|_g^2+(1/2)\int_Q s\langle\nabla_g\Delta_g\varphi,\nabla_g (|z|^2)\rangle-\Re\left(\int_\Upsilon s\partial_{\nu_g}z\Delta_g\varphi\bar{z}\right)
\\
&=\int_{Q} s\Delta_g\varphi|\nabla_g z|_g^2-(1/2)\int_Q s\Delta_g^2\varphi|z|^2
\\
&\hskip 4cm -\Re\left(\int_\Upsilon s\partial_{\nu_g}z\Delta_g\varphi\bar{z}\right)+(1/2)\int_\Upsilon s\partial_{\nu_g}\Delta_g\varphi|z|^2
\end{align*}
and
\begin{align*}
I_6&=\Re\left(-i\int_Q s\varphi'|\nabla_g z|_g^2\right)+\Re\left(-i\int_Q s\langle\nabla_g\varphi',\nabla_g z\rangle \bar{z}\right)+\Re\left(i\int_\Upsilon s\partial_{\nu_g}z\varphi'\bar{z}\right)
\\
&=\Re\left(-i\int_Q s\langle\nabla_g\varphi',\nabla_g z\rangle \bar{z}\right)+\Re\left(i\int_\Upsilon s\partial_{\nu_g}z\varphi'\bar{z}\right).
\end{align*}
Finally, we get
\begin{align*}
I_7&=-\int_{Q} s^3|\nabla_g\varphi|_g^2 \langle\nabla_g\varphi,\nabla_g |z|^2\rangle
\\
&=\int_{Q} s^3\divergence(|\nabla_g\varphi|_g^2 \nabla_g\varphi)|z|^2-\int_\Upsilon s^3|\nabla_g\varphi|_g^2 \partial_{\nu_g}\varphi |z|^2
\\
&=\int_{Q} s^3\left(2\nabla_g^2\varphi(\nabla_g\varphi,\nabla_g\varphi)+|\nabla_g\varphi|_g^2\Delta_g\varphi\right)|z|^2-\int_\Upsilon s^3|\nabla_g\varphi|_g^2 \partial_{\nu_g}\varphi |z|^2,
\end{align*}
which implies
\[
I_7+I_8=\int_{Q} 2s^3\nabla_g^2\varphi(\nabla_g\varphi,\nabla_g\varphi)|z|^2-\int_\Upsilon s^3|\nabla_g\varphi|_g^2 \partial_{\nu_g}\varphi |z|^2.
\]
Adding all of these together gives
\begin{align*}
&\Re (P_s^+ z,\overline{P_s^- z})_g+\mathcal{B}-\int_{Q}s \Re\left(iz\langle\nabla_g\varphi',\nabla_g \bar{z}\rangle\right)+\int_Q s\Re\left(i\langle\nabla_g\varphi',\nabla_g z\rangle \bar{z}\right)
\\
&\hskip 1cm =\int_{Q}2s\Re\left(\nabla_g^2\varphi(\nabla_g z,\nabla_g \bar{z})\right)
\\
&\hskip 3cm +\int_{Q}\left[2s^3\nabla_g^2\varphi(\nabla_g\varphi,\nabla_g\varphi)+(s/2)\varphi''-(s/2)\Delta_g^2\varphi\right]|z|^2,
\end{align*}
where
\begin{align*}
&\mathcal{B}:=-\int_\Upsilon s \Re\left(iz\partial_{\nu_g}\varphi\bar{z}'\right)+\int_\Upsilon 2s\Re\left(\partial_{\nu_g}z\langle\nabla_g\varphi,\nabla_g\bar{z}\rangle\right)-s\int_\Upsilon\partial_{\nu_g}\varphi|\nabla_g z|_g^2
\\
&\hskip 2cm +\int_\Upsilon s\Re\left(\partial_{\nu_g}z\Delta_g\varphi\bar{z}\right)-(1/2)\int_\Upsilon s\partial_{\nu_g}\Delta_g\varphi|z|^2
\\
&\hskip 4cm -\int_\Upsilon s\Re\left(i\partial_{\nu_g}z\varphi'\bar{z}\right)+\int_\Upsilon s^3|\nabla_g\varphi|_g^2 \partial_{\nu_g}\varphi |z|^2.
\end{align*}

Henceforth, $C=C(g,\phi,T)>0$, $\gamma_\ast=\gamma_\ast(g,\phi,T)$ and $s_\ast=s_\ast(g,\phi,T)>0$ denote generic constants, and set $\delta:=\inf_{\{\phi> 0\}}|\nabla_g\phi |_g>0$. Then we have
\begin{align*}
2s\Re\left(\nabla_g^2\varphi(\nabla_g z,\nabla_g \bar{z})\right)&=2\sigma\left(\nabla_g^2\phi(\nabla_g z,\nabla_g \bar{z})+\gamma|\langle\nabla_g\phi,\nabla_g z\rangle|^2\right)
\\
&\ge 2\sigma \nabla_g^2\phi(\nabla_g z,\nabla_g \bar{z})
\\
&= 2\sigma \left(\nabla_g^2\phi(\nabla_g (\Re z),\nabla_g (\Re z))+\nabla_g^2\phi(\nabla_g (\Im z),\Im\nabla_g (\Im z))\right)
\\
&\ge C\sigma|\nabla_g z|_g^2,
\end{align*}
and, by $|\varphi''|\le C\ell^3 e^{2\gamma(\phi+2m)}\le C\xi^3$ and $|\Delta_g^2\varphi|\le C\gamma^4\xi$,
\begin{align*}
&2s^3\nabla_g^2\varphi(\nabla_g\varphi,\nabla_g\varphi)+(s/2)\varphi''-(s/2)\Delta_g^2\varphi
\\
&\hskip 1cm \ge 2\sigma^3 \left(\nabla_g^2\phi(\nabla_g\phi,\nabla_g\phi)+\gamma|\nabla_g\phi|_g^4\right)-Cs\xi^3-Cs\gamma^4\xi
\\
&\hskip 1cm \ge 2\gamma\sigma^3\delta^4-C(\sigma^3+s\xi^3+s\gamma^4\xi^3)
\\
&\hskip 1cm \ge C\gamma\sigma^3,\quad \gamma\ge \gamma_\ast,\; s\ge s_\ast.
\end{align*}
In consequence, we get
\[
C\int_{Q} (\sigma|\nabla_g z|_g^2+\gamma\sigma^3|z|^2)\le \Re(P_s^+z,\overline{P_s^-z})_g+\int_{Q}2s |z||\nabla_g\varphi'|_g|\nabla_g z|_g+\mathcal{B}.
\]

Set
\begin{align*}
&\mathcal{B}_{\Sigma_0}:=\int_{\Sigma_0} \left[-s\Re\left(iz\partial_{\nu_g}\varphi\bar{z}'\right)+2s\Re\left(\partial_{\nu_g}z\langle\nabla_g\varphi,\nabla_g\bar{z}\rangle\right)-s\partial_{\nu_g}\varphi|\nabla_g z|_g^2\right.
\\
&\hskip .5cm \left.+s\Re\left(\partial_{\nu_g}z\Delta_g\varphi\bar{z}\right)-s\Re\left(i\partial_{\nu_g}z\varphi'\bar{z}\right)-(s/2)\partial_{\nu_g}\Delta_g\varphi|z|^2+s^3|\nabla_g\varphi|_g^2 \partial_{\nu_g}\varphi |z|^2\right]
\end{align*}
and
\begin{align*}
&\mathcal{B}_\Sigma:=\int_{\Sigma} \left[-s\Re\left(i z\partial_{\nu_g}\varphi\bar{z}'\right)+2s\Re\left(\partial_{\nu_g}z\langle\nabla_g\varphi,\nabla_g\bar{z}\rangle\right)-s\partial_{\nu_g}\varphi|\nabla_g z|_g^2\right.
\\
&\hskip .5cm \left.+s\Re\left(\partial_{\nu_g}z\Delta_g\varphi\bar{z}\right)-s\Re\left(i\partial_{\nu_g}z\varphi'\bar{z}\right)-(s/2)\partial_{\nu_g}\Delta_g\varphi|z|^2+s^3|\nabla_g\varphi|_g^2 \partial_{\nu_g}\varphi |z|^2\right]
\end{align*}
so that
\[
\mathcal{B}=\mathcal{B}_{\Sigma_0}+\mathcal{B}_\Sigma.
\]
By using $|\Delta_g\varphi|\le C\gamma^2\xi$, $|\varphi'|\le C\xi^2$ and $|\partial_{\nu_g}\Delta_g\varphi|\le C\gamma^3\xi$, we obtain
\[
C\mathcal{B}_{\Sigma_0}\le \int_{\Sigma_0}\left[\sigma^{-1}|z'|^2+\sigma|\nabla_g z|_g^2+\sigma^3|z|^2\right].
\]

On the other hand,  we have
\[
|\nabla_g z|_g^2=|\nabla_{\tau_g}z|_g^2+|\partial_{\nu_g} z|^2,\quad \langle\nabla_g\varphi,\nabla_g \bar{z}\rangle=-\gamma\xi|\nabla_g\phi|_g\partial_{\nu_g} \bar{z}\quad \mbox{on}\; \Sigma
\]
and $\nu_g=-\frac{\nabla_g\phi}{|\nabla_g\phi|_g}$ on $\Sigma$. Whence
\begin{align*}
\mathcal{B}_\Sigma&=\int_{\Sigma} \left[\sigma|\nabla_g\phi|_g\Re\left(i z\bar{z}'\right)-2\sigma|\nabla_g\phi|_g|\partial_{\nu_g}z|^2+\sigma|\nabla_g\phi|_g|\nabla_g z|_g^2\right.
\\
&\hskip 1cm \left.+s\Delta_g\varphi\Re\left(\partial_{\nu_g}z\bar{z}\right)-s\varphi'\Re\left(i\partial_{\nu_g}z\bar{z}\right)-(s/2)\partial_{\nu_g}\Delta_g\varphi|z|^2-\sigma^3|\nabla_g\phi|_g^3 |z|^2\right]
\\
&=\int_{\Sigma} \left[\sigma|\nabla_g\phi|_g\Re\left(i z\bar{z}'\right)-\sigma|\nabla_g\phi|_g|\partial_{\nu_g}z|^2+\sigma|\nabla_g\phi|_g|\nabla_{\tau_g} z|_g^2\right.
\\
&\hskip 1cm \left.+s\Delta_g\varphi\Re\left(\partial_{\nu_g}z\bar{z}\right)-s\varphi'\Re\left(i\partial_{\nu_g}z\bar{z}\right)-(s/2)\partial_{\nu_g}\Delta_g\varphi|z|^2-\sigma^3|\nabla_g\phi|_g^3 |z|^2\right],
\end{align*}
which means
\begin{align*}
&\mathcal{B}_\Sigma+(\delta/2)\int_\Sigma \sigma|\partial_{\nu_g}z|^2
\\
&\hskip 1cm \le\int_\Sigma \left[-(\delta/2)\sigma|\partial_{\nu_g}z|^2+s\left(|\Delta_g\varphi|+|\varphi'|\right)|\partial_{\nu_g}z||z|+\sigma|\nabla_g\phi|_g |z||z'|\right.
\\
&\hskip 3cm \left.+\sigma|\nabla_g\phi|_g|\nabla_{\tau_g} z|_g^2-(s/2)\partial_{\nu_g}\Delta_g\varphi|z|^2-\sigma^3|\nabla_g\phi|_g^3 |z|^2\right].
\end{align*}
We note that
\begin{align*}
&-(\delta/2)\sigma|\partial_{\nu_g}z|^2+s\left(|\Delta_g\varphi|+|\varphi'|\right)|\partial_{\nu_g}z||z|
\\
&\hskip 2cm =-(\delta/2)\sigma\left(|\partial_{\nu_g}z|-s\frac{|\Delta_g\varphi|+|\varphi'|}{\delta\sigma}|z|\right)^2+s^2\frac{\left(|\Delta_g\varphi|+|\varphi'|\right)^2}{2\delta\sigma}|z|^2
\\
&\hskip 2cm \le Cs\gamma^3\xi^3|z|^2.\end{align*}
By $2\sigma|z||z'|\le (s\gamma^{-1}\xi^{-1})|z'|^2+(s\gamma^3\xi^3)|z|^2$, we obtain
\begin{align*}
&\mathcal{B}_\Sigma+(\delta/2)\int_\Sigma \sigma|\partial_{\nu_g}z|^2
\\
&\hskip 1cm \le\int_\Sigma \left[C(\sigma|\nabla_{\tau_g}z|_g^2+s\gamma^{-1}\xi^{-1}|z'|^2+s\gamma^3\xi^3|z|^2)-(\delta^3/2)\sigma^3|z|^2\right],\quad s\ge s_\ast
\end{align*}
and then
\[
\mathcal{B}_\Sigma+(\delta/4)\int_\Sigma \sigma\left[|\partial_{\nu_g}z|^2+\delta^2\sigma^2|z|^2\right]\le C\int_\Sigma (s\gamma^{-1}\xi^{-1}|z'|^2+\sigma|\nabla_{\tau_g}z|_g^2),\quad s\ge s_\ast.
\]

Combining the estimates above, we get 
\begin{align*}
&C\left(\int_{Q}(\sigma|\nabla_g z|_g^2+\gamma\sigma^3|z|^2)+\int_{\Sigma}\sigma(|\partial_{\nu_g}z|^2+\sigma^2|z|^2)\right)
\\
&\hskip 1cm \le  \Re(P_s^+z,\overline{P_s^-z})_g+\int_{Q}2s |z||\nabla_g\varphi'|_g|\nabla_g z|_g
\\
&\hskip 5cm+\mathcal{B}_{\Sigma_0}+C\int_\Sigma (s\gamma^{-1}\xi^{-1}|z'|^2+\sigma|\nabla_{\tau_g}z|_g^2)
\\
&\hskip 1cm \le (1/2)\|P_s z\|_g^2+C\int_{Q}\gamma\xi \left(|\nabla_g z|_g^2+s^2\xi^2|z|^2\right)
\\
&\hskip 3cm+C\int_{\Sigma_0}\left(\sigma^{-1}|z'|^2+\sigma|\nabla_g z|_g^2+\sigma^3|z|^2\right)
\\
&\hskip 5cm +C\int_\Sigma (s\gamma^{-1}\xi^{-1}|z'|^2+\sigma|\nabla_{\tau_g}z|_g^2),\quad s\ge s_\ast,
\end{align*}
where we used $|\nabla_g\varphi'|_g\le C\gamma\xi^2$. As the second term in the right-hand side  can absorbed by the left-hand side, we find
\begin{align*}
&C\left(\int_{Q}\sigma(|\nabla_g z|_g^2+\gamma\sigma^2|z|^2)+\int_{\Sigma}\sigma(|\partial_{\nu_g}z|^2+\sigma^2|z|^2)\right)
\\
&\hskip 2cm \le \|P_s z\|_g^2+\int_{\Sigma_0}\left(\sigma^{-1}|z'|^2+\sigma|\nabla_g z|_g^2+\sigma^3|z|^2\right)
\\
&\hskip 5cm +\int_\Sigma (s\gamma^{-1}\xi^{-1}|z'|^2+\sigma|\nabla_{\tau_g}z|_g^2),\quad s\ge s_\ast.
\end{align*}
Since $u=e^{-s\varphi}z$ and $\nabla_{\tau_g}u=\nabla_{\tau_g}z$ holds by $\nabla_{\tau_g}\phi=0$, we end up getting
\begin{align*}
&C\left(\int_{Q}e^{2s\varphi}\sigma(|\nabla_g u|_g^2+\gamma\sigma^2|u|^2)+\int_{\Sigma}e^{2s\varphi}\sigma(|\partial_{\nu_g}u|^2+\sigma^2|u|^2)\right)
\\
&\hskip 1cm \le \int_Q e^{2s\varphi}|(i\partial_t+\Delta_g)u|^2+\int_{\Sigma_0}e^{2s\varphi}\left(\sigma^{-1}|u'|^2+\sigma|\nabla_g u|_g^2+\sigma^3|u|^2\right)
\\
&\hskip 2cm +\int_\Sigma e^{2s\varphi}(s\gamma^{-1}\xi^{-1}|u'|^2+\sigma|\nabla_{\tau_g}u|_g^2+Cs^3\gamma^{-1}\xi^3|u|^2),\quad s\ge s_\ast.
\end{align*}
As the last zeroth-order term in the right-hand side can absorbed by the left-hand side. The proof is then complete.

\end{proof}

\appendix

\section{Analysis of the IBVP}\label{appA}

Let $A: D(A)\subset L^2(U)\rightarrow L^2(U)$ be the unbounded operator given by
\[
Au:=i(1/\sqrt{|g|})(\partial_k+i a_k)\left(\sqrt{|g|}g^{k\ell}(\partial_\ell +ia_\ell)u\right),\quad  D(A):=H_0^1(U)\cap H^2(U),
\]
where $(a_1,\ldots,a_n)\in W^{1,\infty}(\mathbb{R}^n;\mathbb{R}^n)$.
 
 We verify that $A$ (densely defined) is skew-adjoint. Therefore, it follows from \cite[Proposition 3.7.2]{Tucsnak2009} that $A$ is $m$-dissipative. Then, according to Lumer-Phillips theorem, $A$ generates a contraction semigroup (e.g. \cite[Theorem 3.8.4]{Tucsnak2009}).

Let 
\[
\mathcal{X}_0=C^1([0,T];L^2(U))\cap C([0,T];H_0^1(U)\cap H^2(U)).
\]

As $A$ is the generator of a contraction semigroup $e^{tA}$, for all $v_0\in D(A)$ and $F\in W^{1,1}((0,T);L^2(U))$ the IBVP
\[
\begin{cases}
(\partial_t -iL)v=F\quad \mbox{in}\; U\times (0,T),
\\
v_{|\Sigma}=0,
\\
v(\cdot,0)=v_0
\end{cases}
\]
has a unique solution $v=v(v_0,F)\in \mathcal{X}_0$.

Let us assume now that $(a_1,\ldots,a_n)\in W^{3,\infty}(\mathbb{R}^n;\mathbb{R}^n)$, $A v_0\in D(A)$, $F\in W^{2,1}((0,T),L^2(U))$ and $F(\cdot,0)\in D(A)$. By Duhamel's formula
\[
v(\cdot,t)=e^{tA}v_0+\int_0^te^{sA}F(\cdot, t-s)ds,\quad t\ge 0.
\]
Therefore, we have
\[
\partial_tv(\cdot,t)=e^{tA}(Av_0+F(\cdot,0))+\int_0^te^{sA}\partial_t F(\cdot, t-s)ds,\quad t\ge 0.
\]
In particular, $\partial_tv$ is the solution of the IBVP
\[
\begin{cases}
(\partial_t -iL)v=\partial_tF\quad \mbox{in}\; U\times (0,T),
\\
v_{|\Sigma}=0,
\\
v(\cdot,0)=Av_0+F(\cdot,0)
\end{cases}
\]
and hence $\partial_t v\in \mathcal{X}_0$. Thus, $v\in \mathcal{X}$.

Next, consider the IBVP
\begin{equation}\label{exIBVP_A}
\begin{cases}
(\partial_t-iL)u=0\quad \text{in}\; U\times(0,T),
\\
u_{|\Sigma}=f,
\\
u(\cdot,0)=u_0.
\end{cases}
\end{equation}

Let $(f,u_0)\in \mathcal{T}$ and $H\in \mathcal{H}_0$ be chosen arbitrarily such that $f=H_{|\Sigma}$ and $u_0=H(\cdot,0)$. Since $F:=-(\partial_t -iL)H\in W^{2,1}((0,T); L^2(U))$ and $F(\cdot, 0)\in D(A )$, the last case, with $v_0=0$, shows that $u=H+v(0,F)\in \mathcal{X}$ is a solution of the IBVP \eqref{exIBVP_A}. The previous analysis also shows that $u=0$ is the unique solution of the IBVP \eqref{exIBVP_A} when $f=0$ and $u_0=0$.

\section*{Acknowledgement}

This work was supported by JSPS KAKENHI Grant Number JP23KK0049.

\section*{Declarations}
\subsection*{Conflict of interest}
The authors declare that they have no conflict of interest.

\subsection*{Data availability}
Data sharing not applicable to this article as no datasets were generated or analyzed during the current study.

\bibliographystyle{plain}
\bibliography{Schrodinger_continuation_2}

\begin{thebibliography}{1}

\bibitem{Baudouin2002}
L.~Baudouin and J.-P. Puel.
\newblock Uniqueness and stability in an inverse problem for the schrödinger
  equation.
\newblock {\em Inverse Problems}, 18:1537--1554, 2002.

\bibitem{Choulli}
M.~Choulli.
\newblock {\em An Introduction to the Uniqueness of Continuation of Second
  Order Partial Differential Equations}.
\newblock to appear.

\bibitem{CT2024a}
M.~Choulli and H.~Takase.
\newblock An inverse hyperbolic obstacle problem.
\newblock {\em arXiv:2407.05662}, 2024.

\bibitem{CT2024}
M.~Choulli and H.~Takase.
\newblock Lipschitz stability for an elliptic inverse problem with two
  measurements.
\newblock {\em arXiv:2404.13901}, 2024.

\bibitem{Choulli2021}
M.~Choulli and M.~Yamamoto.
\newblock Global stability result for parabolic cauchy problems.
\newblock {\em Journal of Inverse and Ill-Posed Problems}, 29:895--915, 12
  2021.

\bibitem{Huang2019}
X.~Huang, Y.~Kian, E.~Soccorsi, and M.~Yamamoto.
\newblock Carleman estimate for the schrödinger equation and application to
  magnetic inverse problems.
\newblock {\em Journal of Mathematical Analysis and Applications},
  474:116--142, 2019.

\bibitem{Mercado2023}
A.~Mercado and R.~Morales.
\newblock Exact controllability for a schrödinger equation with dynamic
  boundary conditions.
\newblock {\em SIAM Journal on Control and Optimization}, 61:3501--3525, 2023.

\bibitem{Mercado2008}
A.~Mercado, A.~Osses, and L.~Rosier.
\newblock Inverse problems for the schrödinger equation via carleman
  inequalities with degenerate weights.
\newblock {\em Inverse Problems}, 24(1), 015017:18pp, 2008.

\bibitem{Tucsnak2009}
M.~Tucsnak and G.~Weiss.
\newblock {\em Observation and Control for Operator Semigroups}.
\newblock Birkhäuser Verlag, 2009.

\end{thebibliography}

\end{document}